\documentclass{amsart}
\usepackage[nobysame]{amsrefs}
\usepackage{tikz}
\usepackage{graphicx}
\usepackage{pict2e}
\usepackage{setspace}
\usepackage{amssymb}
\usepackage{fullpage}
\usepackage{comment}
\usepackage[T1]{fontenc}
\usepackage{setspace}
\newtheorem{proposition}{Proposition}
\newtheorem{theorem}{Theorem}
\newtheorem{lemma}{Lemma}
\newtheorem{definition}{Definition}

\usepackage{color}

\newcommand\ben{\begin{enumerate}}
\newcommand\een{\end{enumerate}}
\newcommand\bit{\begin{itemize}}
\newcommand\eit{\end{itemize}}

\title{Equilateral Triangle Skew Condition for Quasiconformality}
\author{Colleen Ackermann}
\thanks{C.A. was supported in part by a gift to the Mathematics Department at the University of Illinois from Gene H. Golub.}
\address{Colleen Ackermann, St. Mary's College of Maryland,
Department of Mathematics and Computer Science,
47645 College Drive,
St. Mary's City, Maryland 20686, United States}
\email{ctackermann@smcm.edu}

\author{Peter Ha\"{i}ssinsky}
\thanks{P.H. is partially supported by the  ANR projects ``GDSous/GSG'' no. 12-BS01-0003-01
and ``Lambda'' no. 13-BS01-0002}
\address{Peter Ha\"{\i}ssinsky, Universit\'e Paul Sabatier,
Institut de Math\'ematiques de Toulouse,
118 route de Narbonne,
31062 Toulouse Cedex 9, France}
\curraddr{
Aix Marseille Univ, CNRS, Centrale Marseille, I2M, Marseille, France}
\email{phaissin@math.cnrs.fr}

\author{Aimo Hinkkanen}
\address{Aimo Hinkkanen, University of Illinois at Urbana-Champaign,
Department of Mathematics,
1409 W. Green Street,
Urbana, Illinois 61801, United States}
\email{aimo@math.uiuc.edu}

\date{\today}

\begin{document}

\maketitle

\begin{abstract} We characterize quasiconformal mappings in terms of the distortion of the vertices of
equilateral triangles.\end{abstract}

\section{Introduction}

Since quasiconformal mappings were first studied nearly a century ago, many diverse characterizations have been discovered.  These have led to a wide variety of applications in many fields including Teichm\"{u}ller theory, elliptic PDE's, hyperbolic geometry  and complex dynamics.  For an overview of these applications and the theory of quasiconformal mappings see \cite{MR0200442}, \cite{MR2472875}, and \cite{MR2245223}.  In this paper we will use the metric definition of quasiconformality to obtain a new formulation for the definition of planar quasiconformal mappings.  

\begin{definition}\label{metricdefn}
Let $f:U\to V$ be a homeomorphism between planar domains.  When $D(z,r)=\{w\in\mathbb{C}:|z-w|\leq r\}\subset U$, define
$$ M(z,r)=\sup \{|f(z)-f(w)| : |z-w|=r\}, \text{ and}$$
$$ m(z,r)=\inf \{|f(z)-f(w)| : |z-w|=r\}.$$ 
Suppose that $K\geq 1$. 
The mapping $f$ is said to be $K$-quasiconformal if 
$$H(z)=\limsup_{r\to 0}\frac{M(z,r)}{m(z,r)}\leq K$$
for a.e.~$z\in U$, and if $H(z)$ is bounded in $U$.
\end{definition}

In his book \cite{MR2245223} John Hubbard obtained a new characterization of quasiconformal mappings.  Let $T$ be a closed topological triangle with specified vertices, 
$L(T)=\max\{ |a-b| : a,b \text{ are vertices of } T\}$ and $l(T)=\min\{ |a-b| : a,b \text{ are vertices of } T\}$.  We define
$${\hbox{\rm skew}}(T)=\frac{L(T)}{l(T)}.$$  
Note that $f(T)$ is also a topological triangle so the expression ${\hbox{\rm skew}}(f(T))$ is defined.  
Then $f:U\to V$ is quasiconformal provided that there exists an increasing homeomorphism $h:[0,\infty)\to [0,\infty)$ such that
$$ {\hbox{\rm skew}}(f(T))\leq h({\hbox{\rm skew}}(T))$$
for all closed Euclidean triangles $T\subset U$.
In fact, Hubbard showed that it is sufficient to only consider triangles with skew bounded above by $\sqrt{7/3}$.  He then asked the question of whether it suffices to only consider equilateral triangles.  Progress was made on this problem in a previous paper by Javier Aramayona and Peter Ha\"{i}ssinsky \cite{MR2436734} in which they showed that  there exists a constant $\epsilon_0>0$ such that if $\epsilon\in[0,\epsilon_0)$ and
$${\hbox{\rm skew}}(f(T))\leq 1+\epsilon$$
for all equilateral triangles $T\subset U$, then $f$ is quasiconformal.

Theorem \ref{thm:main} of this paper answers Hubbard's question in the affirmative. 

\begin{theorem}\label{thm:main} Let $U$ and $V$ be two domains in the complex plane $\mathbb{C}$, and
let $f: U \to V$ be an orientation-preserving homeomorphism.  For each $\sigma\geq 1$ there exists $H(\sigma)\geq 1$ with the following property.  If 
${\hbox{\rm skew}}(f(T))\leq \sigma$ for all equilateral triangles $T\subset U$,
then, 
for every $z\in U$ and every  $r< {\mbox{\rm dist}}(z,{\mbox{$\mathbb C$}}\setminus U)$,  the inequality  $M(z,r) \le H m(z,r)$  holds where $H=H(\sigma)$. In particular, the map  $f$ is quasiconformal.
\end{theorem}

Our next result, Theorem \ref{cor:1}, serves two purposes.  On the one hand, the value $H$ from Theorem \ref{thm:main}
gives an upper bound for the dilatation of $f$.  However, this estimate is far from optimal.  We may use the fact that quasiconformal
mappings  are differentiable almost everywhere to obtain a sharp bound on the dilatation in terms of $\sigma$. On the other hand,
quasiconformality is a local property and demanding, as in Theorem \ref{thm:main}, a uniform bound on every equilateral triangle is unnecessarily strong.
 
To show that $f$ is quasiconformal in  $U$, it suffices to show that $f$ satisfies two conditions.  First we must show  
that each $z\in U$ has a neighborhood $W$ such that ${\hbox{\rm skew}}(f(T))$ is bounded for all equilateral triangles $T\subset W$. 
If the resulting upper bounds vary, this is not  enough to get the result that $f$ is quasiconformal in $U$. However, by formulating
a new concept of skew that needs to be satisfied only almost everywhere and must be bounded uniformly in $U$, and combining it with
the assumption that the skew as defined previously is bounded, not necessarily uniformly, in some neighborhood of each point of $U$,
we obtain both quasiconformality and a better upper bound for the maximal dilatation of $f$. 
The details of how to define this new notion of skew are given below.

First note that quasiconformal maps are differentiable almost everywhere by Mori's theorem \cite{MR0083024}. 
Let ${\hbox{\rm Skew}}(f)$ denote the supremum of ${\hbox{\rm skew}}(f(T))$ over all equilateral triangles $T$ contained in $U$. 
Next, for $z\in U$ and $r>0$, let ${\hbox{\rm skew}}(f,z,r)$ 
denote the least upper bound of ${\hbox{\rm skew}}(f(T))$ over all equilateral triangles $ T\subset \{w\in U:|z-w|< r\}$.  Set ${\hbox{\rm skew}}(f,z)=\liminf_{r\to 0} {\hbox{\rm skew}}(f,z,r)$ and
${\hbox{\rm skew}}(f)=|| {\hbox{\rm skew}} (f,z)||_{\infty}$. Below, the assumption that ${\hbox{\rm Skew}}(f|W)$ is finite means that there is {\sl some} upper bound for ${\hbox{\rm skew}}(f(T))$ over all equilateral triangles $T$ contained in the appropriate domain $W$.

\begin{theorem}\label{cor:1} Let $U$ be a domain in the complex plane $\mathbb{C}$, and
let $f: U \to f(U)$ be an orientation-preserving homeomorphism.
Suppose that each $z\in U$ has a neighborhood $W$ such that ${\hbox{\rm Skew}}(f|W)$
is finite. 

Suppose that $\sigma \geq 1$.  If ${\hbox{\rm skew}}(f)\leq \sigma$  then $f$ is $K(\sigma)$-quasiconformal where
$$K(\sigma)=\frac{\sigma^2-1+\sqrt{\sigma^4+\sigma^2+1}}{\sqrt{3}\sigma}\,.$$
In particular, if ${\hbox{\rm skew}}(f)=1$  then $f$ is a conformal mapping. The upper bound $K(\sigma)$ for the maximal dilatation of $f$ is the best possible and is attained at least for certain affine mappings. 
\end{theorem}

\section{Proof of the Main Theorem}

Throughout the rest of the paper we will use the following notation and conventions:
\begin{enumerate}
\item We define $D(z,r)=\{w\in\mathbb{C}:|z-w|\leq r\}$ and let $C(z,r)$ be the boundary of $D(z,r)$.  
\item By a curve we mean the image of a not necessarily one-to-one continuous function from a closed interval into $\mathbb{C}$.
\item All triangles (without the qualifier ``topological'') will be closed Euclidean triangles.
\item Let $\mathcal{F}_\sigma$ denote the set of orientation-preserving homeomorphisms of any domain $U\subset\mathbb{C}$ into any domain $V\subset\mathbb{C}$ such that ${\hbox{\rm skew}}(f(T))\leq \sigma$ for all closed equilateral triangles $T\subset U$.
\end{enumerate}

Our strategy for the proof of   Theorem \ref{thm:main} 
is to show first that  the image of every equilateral triangle, $T$, contains a disk with radius proportional to $L(f(T))$. To do this, we naturally need the assumption of Theorem \ref{thm:main} for all equilateral triangles (at least in a suitable set containing this $T$) and not only for this particular $T$. Once it is known that the images of equilateral triangles are ``thick'' in this sense, it is easier to  obtain the quantitative estimates that are required to prove that $f$ is quasiconformal, whether using the metric (Definition \ref{metricdefn}) or the analytic (Definition \ref{def2} in Section \ref{ansection} below) definition of quasiconformality.

The most difficult part of the proof that the image $f(T)$ is thick is to construct another equilateral triangle with suitable additional properties, to which the assumption of  Theorem \ref{thm:main}  can then be applied. This will be done in the proof of Proposition \ref{prop:main2} below. That proof requires delicate geometric considerations. After these strategic comments, let us now move to the proof.

The proof of Theorem \ref{thm:main} relies on the following proposition.

\begin{proposition}\label{prop:main1} 
Let $U$ be a domain containing $D(0,1)$, let $f:U\to\mathbb{C}$ belong to ${\mathcal F}_\sigma$, 
and let $T$ be the triangle with vertices $0$, $1$, and $\omega=1/2+(\sqrt{3}/2)i$.
Then there exists a disk $D$ contained in $f(T)$ such that
\begin{enumerate}
\item $D$ is centered at $f(p)$ where $p=1/2+(85\sqrt{3}\cdot 2^{-9})i\approx 0.5+0.29i$, and
\item there exists a constant $\alpha=\alpha(\sigma)$ such that the radius of $D$ is at least $\alpha L(f(T))$.
\end{enumerate}
\end{proposition}
We note that if $f$ is to be quasiconformal, then, certainly, the image $f(T)$ has to contain a disk of a definite size
 centered at the image of the centroid of the triangle, i.e., the point $\xi= 1/2 + (\sqrt{3}/6)i$. Unfortunately, 
its arithmetic properties make it difficult to relate this point $\xi$ to
the vertices of $T$ using equilateral triangles. The point $p$ was chosen, because it is both close to the centroid
 ($|\xi-p|= \sqrt{3}/(2^9\cdot3)$), and it is a vertex of a tiling of the plane by equilateral triangles whose vertices include the vertices of
 $T$.  Indeed, we have 
$ p=  1/2- 85\cdot 2^{-9} + 85\cdot 2^{-8}\omega$, cf.\ Lemma \ref{lem:lowerbound}. We finally observe that $p$ is closer to the horizontal side of $T$ than $\xi$.

\medskip

We first derive the proof of Theorem \ref{thm:main} from Proposition \ref{prop:main1}. We will then focus on the proof of the latter.

\begin{proof}[Proof of Theorem \ref{thm:main}]

Fix $z\in U$ and $r>0$.  If $D(z,r)\subset U$, let $M(z,r)=\max \{|f(z)-f(w)|: w\in C(z,r)\}$ and $m(z,r)=\min \{|f(z)-f(w)|: w\in C(z,r)\}$.  
Denote by $z_M$ a point in $C(z,r)$ such that $|f(z_M)-f(z)|=M(z,r)$.

Since ${\mathcal F}_\sigma$ is invariant under pre- and post-composition by affine maps of the form $z\mapsto az+b$ where $a,b\in {\mathbb C}$ and $a\not= 0$, we assume that $z=0$,  $r=1$, and $z_M=1$. Thus $|f(0)-f(1)| = M(0,1)$. Let $T_1$ be the equilateral triangle with vertices $0$, $1$, and $\omega$.   Since $0$ and $1$ are vertices of $T_1$, we have $|f(0)-f(1)| \leq L(f(T_1))$. Hence $M(0,1)  \leq L(f(T_1))$.
  
By Proposition \ref{prop:main1}, the image of $T_1$ must contain a disk $D$ centered at $f(p)$ and of radius at least $\alpha L(f(T_1))$.  

Let us consider the isometry $A(z)= \overline{z-p}$. Let $T_2=A(T_1)$ (see Figure~\ref{CircleAndTriangle}).  The triangle $T_2$ is contained in the unit disk, and $A$ maps $p$
to $0$ and $1$ to $p$.  Since the vertices of $T_2$ other than $p$ lie outside of $T_1$, their images lie  outside of $D$, and so we have $L(f(T_2))\ge  \alpha L(f(T_1))$.

Moreover, another application of Proposition \ref{prop:main1} implies that $f(T_2)$ contains the disk $D(f(0),\alpha L(f(T_2)))$. Since all the points of $f(C(0,1))$ are outside the interior of $f(T_2)$ and thus also outside the interior of $D(f(0),\alpha L(f(T_2)))$, it follows that $m(0,1)\geq \alpha L(f(T_2))$.

Summing up these estimates, we obtain
$$m(0,1)\ge \alpha L(f(T_2)) \ge \alpha^2 L(f(T_1)) \ge  \alpha^2 M(0,1)\,.$$
\end{proof}

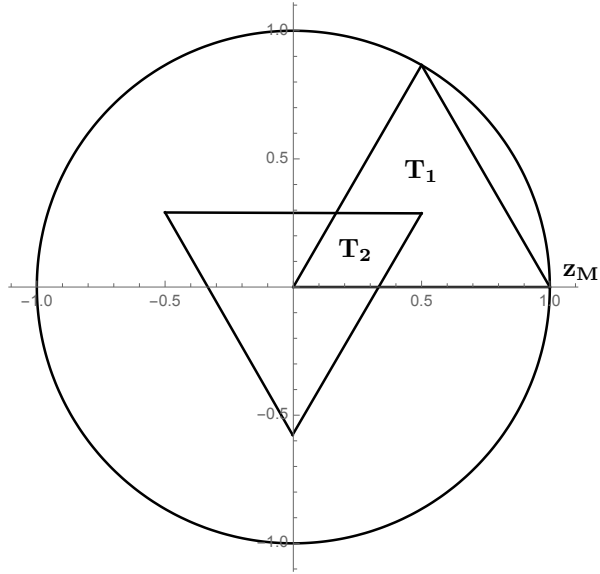
\begin{figure}[h] 
\centering
 \begin{picture}(250,250)
 \put(150,150){${\bf T_1}$}
 \put(125,120){${\bf T_2}$}
 \put(210,113){${\bf z_M}$}
\put(0,0){\includegraphics[width=3 in]{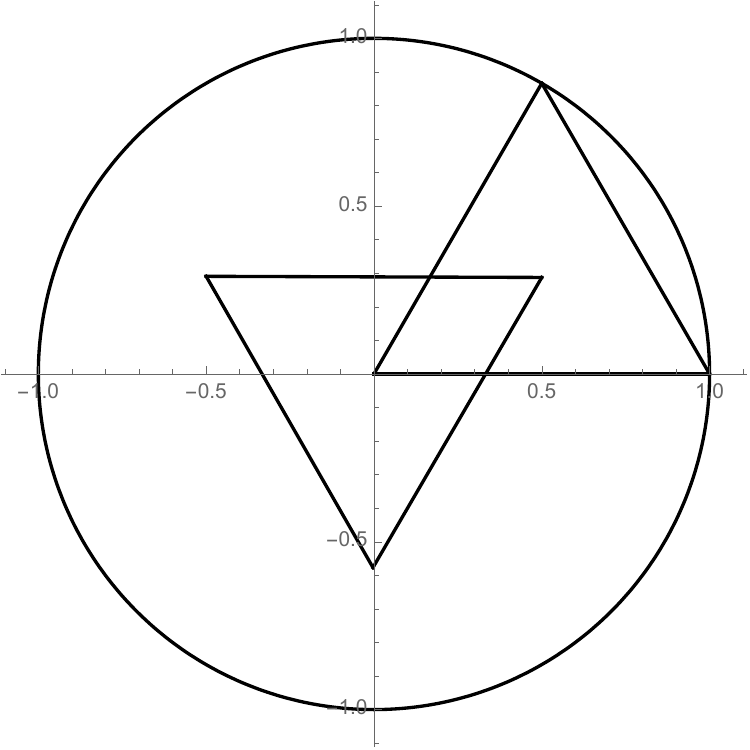}}
\end{picture}
\caption{Configuration of $C(0,1)$, $T_1$ and $T_2$}
\label{CircleAndTriangle}
\end{figure}

\section{Construction of Certain Triangles}

Proposition \ref{prop:main1} is a consequence of the following proposition.
\begin{proposition}\label{prop:main2} 
Let $U$ be a neighborhood of $D(0,1)$, and let $f:U\to\mathbb{C}$ be a homeomorphism onto its image such that $f\in{\mathcal F}_\sigma$. 
Let $T$ be the closed triangle with vertices $0$, $1$, and $\omega=1/2+(\sqrt{3}/2)i$. Let $q=p +2^{-9}$.
Then there exist points $t_1,t_2\in T$ such that  the points $q, t_1, t_2$ form the vertices of an equilateral triangle and
the inequalities 
$|f(t_j)-f(p)|\leq C \mu$, and $|f(p)-f(q)|\geq c L(f(T))$ hold
for some constants $c=c(\sigma)$ and  $C=C(\sigma)$ where $\mu= {\mbox{\rm dist}}(f(p),{\mbox{$\mathbb C$}}\setminus f(T))$.  We permit the trivial triangle where we have $t_1=t_2=q$.
\end{proposition}

\begin{proof} [Proof of Proposition \ref{prop:main1} assuming Proposition \ref{prop:main2}]
If $t_1=t_2=q$, then we have
$$ c L(f(T))\leq |f(p)-f(q)|\leq C\mu.$$
Otherwise, by the triangle inequality we have
\begin{equation} \label{tri1}
|f(p)- f(q)| \leq  |f(t_1)-f(p)| + |f(t_1)- f(q)|  .
\end{equation}
Since $f \in {\mathcal F}_{\sigma}$ and $t_1,t_2,q$ are the vertices of an equilateral triangle, we obtain
\begin{equation} \label{tri2}
 |f(t_1)- f(q)|\leq \sigma  |f(t_1)-f(t_2)|  .
\end{equation}
By Proposition \ref{prop:main2}, we have
$ |f(t_1)-f(p)|  \leq C \mu$ and $ |f(t_2)-f(p)|  \leq C \mu$, so that by this and the triangle inequality we get
$$
|f(t_1)-f(t_2)|  \leq |f(t_1)-f(p)|  +   |f(t_2)-f(p)| \leq 2 C \mu  .
$$
By Proposition \ref{prop:main2}, we have $|f(p)-f(q)|\geq c L(f(T))$. Using this and appealing to $ |f(t_1)-f(p)|  \leq C \mu$ again, we obtain from (\ref{tri1}) and (\ref{tri2}) that
$$
c L(f(T)) \leq |f(p)-f(q)| \leq C \mu + \sigma ( 2 C \mu)  ,
$$
whence
$$
 \mu \ge \frac{c}{(2\sigma+1)C}L(f(T))\,.
 $$
 \end{proof}

\section{Proof of Proposition \ref{prop:main2}}
  
Let the assumptions of  Proposition \ref{prop:main2} be satisfied. Recall that we write $\mu= {\mbox{\rm dist}}(f(p),{\mbox{$\mathbb C$}}\setminus f(T))$. The idea of our proof of Proposition \ref{prop:main2} is to define a curve  $\gamma'$ in $T$ going through $p$ such that
\begin{enumerate}
\item  for all $t\in\gamma'$ we have $|f(t)-f(p)|\leq  \sigma \mu (1+2\sigma^3)$;
\item  there are two points $t_1,t_2\in \gamma'$, such that  $q, t_1, t_2$ form the vertices of an equilateral triangle.
\end{enumerate}
The proof of Proposition \ref{prop:main2}  results from Lemma \ref{lem:lowerbound} and Lemma \ref{lem:curve}.

We first prove the following result. Recall that $\omega = 1/2 + (\sqrt{3}/2) i$. 

\begin{lemma}\label{lem:lowerbound} Let $T$  be the closed triangle with vertices $0$, $1$, and $\omega$. Let  $p =1/2+(85\sqrt{3}\cdot 2^{-9})i$ and  $q=p +2^{-9}$.
Then 
$|f(q)-f(p)|\geq c L(f(T))$
for some positive constant $c=c(\sigma)$. \end{lemma}

\begin{proof}
Let us first consider the tiling of the plane by equilateral triangles with vertices in $\Lambda= {\mbox{$\mathbb Z$}}\oplus \omega{\mbox{$\mathbb Z$}}$. Define a chain of triangles
$(T_j)_{0\le j \le J}$ as a sequence of triangles with vertices in $\Lambda$ such that $T_j\cap T_{j+1}$ is an edge for all $j$ with $0\le j< J$.
Given two edges $(v,w)$ and $(v',w')$, we may connect them by a chain of minimal length $n\ge 0$. A simple induction argument 
implies 
\begin{equation} \label{sigman}
|f(v)-f(w)|\le  \sigma^n |f(v')-f(w')|
\end{equation}
 if $f\in{\mathcal F}_\sigma$ is defined in a neighborhood of the chain.

 Let $T$ be as defined in our hypotheses: it is tiled by $N= 2^{18}$ triangles of $2^{-9}\Lambda$,  and $[p,q]$ is an edge of this tiling. Therefore, for every
 other edge $[v,w]$, it follows that $$|f(v)-f(w)|\le \sigma^{N} |f(p)-f(q)|.$$
 But each side of $T$ is the union of less than $N$ edges of our tiling, therefore, the triangle inequality implies
 $$L(f(T))\le N\sigma^N  |f(p)-f(q)|\,.$$\end{proof}

We now prove a  geometric lemma which will be used in the proof of Lemma \ref{lem:curve}.

\begin{lemma}\label{lem:geo} Let $|z|\le 1/8$ and suppose that $|\theta_\pm  - (\pm\pi/3)| \le 1/8$.  Then the angle $\varphi$ between 
$e^{i\theta_+}-z$ and $e^{i\theta_-}-z$ which crosses the positive real axis belongs to $(\pi/3, \pi)$.\end{lemma}

\begin{proof} We note that $\cos \theta_\pm\ge 1/2- 1/4 >1/8\ge |z|$ so that $\varphi$ is less than $\pi$.

For the other inequality, we will estimate $\tan |\arg (e^{i\theta_\pm} -z)|$ to obtain a lower bound of both angles
with the horizontal line: 
\begin{eqnarray*} 
\tan |\arg (e^{i\theta_\pm} -z)| &  \ge &  \frac{\sqrt{3}/2 - (|z|+1/8)}{1/2 + (|z|+1/8)} \ge \frac{\sqrt{3}/2 - 1/4}{1/2 + 1/4}\\
& \ge & \frac{2\sqrt{3} - 1}{3} \ge \frac{2}{3} > \tan (\pi/6).\end{eqnarray*}
Therefore $\varphi$ is at least $\pi/3$. \end{proof}

Now we demonstrate how to find the curve $\gamma'$ mentioned above.

\begin{lemma}\label{lem:curve} Under the assumptions of Proposition \ref{prop:main2}, there exists a curve
 $\gamma'$ in $T$ going through $p$ such that for all $t\in\gamma'$ we have
$$|f(t)-f(p)|\leq  \sigma \mu (1+2\sigma^3)$$
and there are two points $t_1,t_2\in \gamma'$, such that $q, t_1, t_2$ form the vertices of an equilateral triangle.  We permit the trivial triangle where we have $t_1=t_2=q$.
 \end{lemma}
 
 \begin{proof}
 We will do this in several steps. We first define a curve that will join two points of the boundary of a disk contained in $T$ (Step 1). To make sure that we will be able to  find two points that form an equilateral triangle with $q$, we will extend this curve so that it has end points in a slightly larger disk, and is only close to the boundary of the larger disk when it is also close to its end points (Step 2). Then we will use Lemma \ref{lem:geo} to find our triangle (Step 3).

Since $\sqrt{3}\ge 8/5$, it follows that 
${\mbox{\rm dist}}(p,\partial T) = 85\sqrt{3}\cdot 2^{-9} >  1/4 + 2^{-6}$, so that
 $D(p, 1/4+2^{-6} )$ is contained in the interior of $T$.

 Throughout the proof, for $x\in {\mbox{$\mathbb C$}}$, $R_x$ will denote the rotation centered at $x$ by $\pi/3$ radians, defined by $R_x(z)= x + (z-x)\omega$ and
$\bar R_x$  the rotation centered at $x$ by $-\pi/3$ radians, defined by $\bar R_x(z)= x + (z-x)\bar \omega$.  Recall that we set $\omega=1/2+(\sqrt{3}/2)i$.
 \\ \\

{\bf Step 1: There exists a curve $\gamma_2$  that satisfies the following:
\begin{enumerate}
\item $\gamma_2\subset D(p,1/4)$,
\item $\gamma_2$ has end points on $C(p,1/4)$ which are exactly $2\pi/3$ radians apart, and 
\item for all points $t\in\gamma_2$ we have $|f(t)-f(p)|\leq \sigma\mu$.
\end{enumerate} }

Let $p'\in \partial T$ be such that $d(f(p),f(p'))=\mu$ and let $\gamma = f^{-1}([f(p),f(p')])$.
Since $D(p,1/4)$ is contained in the interior of $T$, we may consider  the component  $\gamma_1$ of $\gamma\cap D(p,1/4)$ 
that contains $p$, and we denote by $w\in C(p,1/4)$
the other end point of $\gamma_1$.  We take $w$ to be the first point of $C(p,1/4)$ encountered when moving along $\gamma$ starting from $p$.
\\ \indent Now define $$\gamma_2 = R_p(\gamma_1)\cup \bar R_p(\gamma_1)\,.$$
Note that, for every $s\in \gamma_1$, $R_p(s)$ and $\bar R_p(s)$ are two points in $\gamma_2$ which make an angle of $2\pi/3$ seen
from $p$.  
Since $f\in {\mathcal F}_\sigma$, for all $t\in\gamma_2$, we have
$$|f(t)-f(p)|\leq \sigma|f(s)-f(p)|\leq \sigma\mu$$  where  $s\in \gamma_1$  denotes a point such that either $t=R_p(s)$ or $t=\bar R_p(s)$.
\\ \\

{\bf Step 2: Let $a,b$ be the end points of $\gamma_2$. There exists a curve $\gamma_3$ such that 
\begin{enumerate}
\item $\gamma_3\subset D(p,1/4) \cup D(a,2^{-6})\cup D(b,2^{-6})$; 
\item $\gamma_3$ has both end points on $C(p,1/4+\sqrt{3}\cdot 2^{-7})$;
\item for all points $t\in\gamma_3$, $|f(t)-f(p)|\leq \sigma\mu(1+2\sigma^3).$
\end{enumerate}}

Let $D_a=D(a,2^{-6})$ and $D_b=D(b,2^{-6})$.  
Let $\gamma_{2a}$ and $\gamma_{2b}$ be the components of $\gamma_2\cap D_a$ and $\gamma_2\cap D_b$ that have end points at $a$ and $b$ respectively. 

Clearly $\gamma_{2a}$ also has an end point on the boundary of $D_a$.  Let $a'$ denote an end point of $\gamma_{2a}$ 
on the boundary of $D_a$.  Use the tangent line to $D(p,1/4)$ at $a$ to divide $D_a$ in half, and then divide each half into thirds.  
So we have divided $D_a$ into closed sectors of $\pi/3$ radians with three such sectors lying entirely outside of $D(p,1/4)$.  
Let $S_a$ denote the middle sector lying completely outside of $D(p,1/4)$.  Then there exists $n\in\{2,3\}$ such that when  
$\gamma_{2a}$ is rotated $n\pi/3$ radians in an appropriate direction about $a$, the image of $a'$ under the rotation will lie in $S_a$.  
Let the image of $\gamma_{2a}$ under this rotation be denoted by $\gamma_{3a}$.  

Now we will bound the quantity $|f(t)-f(p)|$ where $t\in\gamma_{3a}$.  Fix $t\in\gamma_{3a}$.    
Let $t_0$ be the point on $\gamma_{2a}$ whose image under the rotation is $t$.  Without loss of generality 
we will assume this rotation was clockwise.  Let $t_i$ denote the image of $t_0$ under a clockwise rotation of $i\pi/3$ radians
where $ i=1, \ldots , n$ ($t=t_n$).  Then since $a, t_{i-1},t_i$ ($1\leq i\leq n$) form an equilateral triangle, we have
$$|f(t_i)-f(a)|\leq \sigma|f(t_{i-1})-f(a)|.$$
Since $a,t_0\in\gamma_2$ we have
$$|f(a)-f(t_0)|\leq |f(a)-f(p)|+|f(p)-f(t_0)|\leq 2\sigma\mu$$
and
$$|f(a)-f(p)|\leq \sigma\mu.$$
Thus since $n$ is at most $3$ we may apply (\ref{sigman}) with $n\leq 3$, which together with the triangle inequality yields
$$|f(t)-f(p)|\leq |f(a)-f(p)|+|f(a)-f(t)|\leq \sigma\mu +\sigma^n|f(a)-f(t_0)|\leq \sigma\mu(1+2\sigma^n)\leq \sigma\mu(1+2\sigma^3).$$

\indent  Furthermore $\gamma_{3a}$ must intersect the circle $C(p,1/4+\sqrt{3}\cdot 2^{-7})$.  This is because $\gamma_{3a}$ 
has an end point in $S_a$ and therefore the distance of the end point of $\gamma_{3a}$ from $D(p,1/4)$ must be at least 
$\cos(\pi/6)\cdot 2^{-6}= \sqrt{3}\cdot 2^{-7}$.  This is depicted in Figure \ref{SACircle}.

\begin{figure}[h]
\centering
 \begin{picture}(250,400)
 \put(10,210){${\bf D_a}$}
 \put(105,152){ ${\bf S_a}$}
 \put(131,80){${\bf D(p,1/4+\sqrt{3}/2^7)}$}
 \put(19,70){${\bf D(p,1/4)}$}
\put(0,0){\includegraphics[width=2 in]{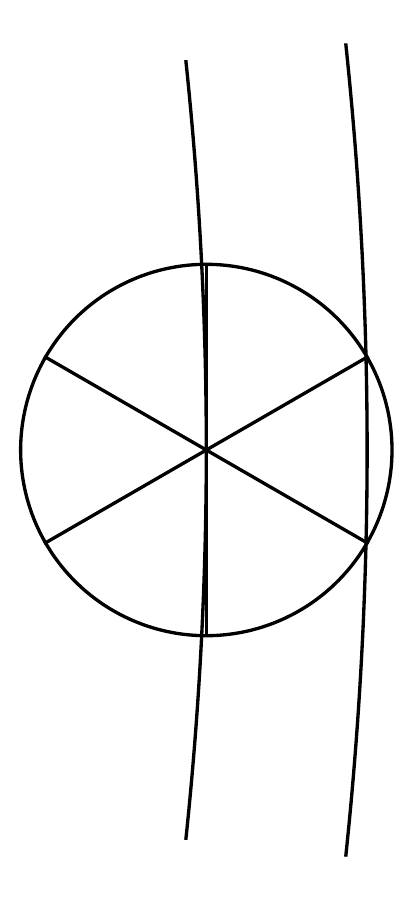}}
\end{picture}
\caption{}
\label{SACircle}
\end{figure}

We proceed similarly near  $b$ and define 
a curve $\gamma_{3b}$ contained in $D_b$ with end points at $b$ and at some point on the intersection of the boundary of 
$D_b$ and $S_b$ (defined analogously to $S_a$) such that for all $t\in\gamma_{3b}$ we have
$|f(t)-f(p)|\leq \sigma\mu(1+2\sigma^3)$; as above,  $\gamma_{3b}$  intersects the circle $C(p,1/4+\sqrt{3}\cdot 2^{-7})$. 

Let $\gamma_3$ be the connected component of $(\gamma_2\cup \gamma_{3a}\cup \gamma_{3b})\cap D(p,1/4+\sqrt{3}\cdot 2^{-7})$ which includes points in both 
$\gamma_{3a}$ and $\gamma_{3b}$.  
Then for all points $t\in\gamma_3$, $$|f(t)-f(p)|\leq \sigma\mu(1+2\sigma^3).$$ 
The curve $\gamma'$ in Lemma \ref{lem:curve} can be chosen as  $\gamma'=\gamma_3$.
\\ 
\\

{\bf Step 3: Let $q= p+2^{-9}$. There exist $t_1,t_2\in \gamma_3$ such that $\{q,t_1,t_2\}$ form an equilateral triangle.}

Let $D_q$ be the smallest disk centered at $q$ which contains $D(p,1/4)$.  Then $D_q\subset D(p, 1/4 +\sqrt{3}\cdot 2^{-7})$ since $|p-q|=2^{-9}\le \sqrt{3}\cdot2^{-8}$.  Let $\gamma_4$ be the connected component of $\gamma_3\cap D_q$ which has end points  
$A\in D_a\cap D_q$ and $B\in D_b\cap D_q$.  

Note that, if we write  $a= p+|a-p|e^{i\theta_a}$ and $A= p +|A-p| e^{i\theta_A}$, where $|\theta_A - \theta_a|$ is chosen to be as small as possible modulo $2\pi$, then $$|\theta_A - \theta_a|\le 2|a-A|/(1/4)\le2\cdot 2^{-6}/(1/4) \le 1/8$$ and similarly for $b$ and $B$. Note that $|\arg(a-p)| + |\arg (b-p)|=2\pi/3$ and $|p-q|/(1/4)= 2^{-7}\le 1/8$. Therefore, 
by Lemma \ref{lem:geo} applied in $D(p,1/4)$, the angle between $A-q$ and $B-q$
lies in $(\pi/3, \pi)$.  Hence, 
the images $A_r$ and $B_r$ of $A$ and $B$ respectively under $\bar R_q$ will separate $A$ and $B$ on $\partial D_q$.
 Thus the image  $\bar R_q(\gamma_4)$  must intersect $\gamma_4$ at a point $q'$, say.  This gives us our desired equilateral triangle since $q$, $q'$,  and the pre-image $q''$ of $q'$ form an equilateral triangle. 
 
 This completes the proof of Lemma \ref{lem:curve}.
\end{proof}

To complete the proof of Proposition \ref{prop:main2}, we note that the points $t_1$ and $t_2$ required will be the points $q'$ and $q''$. It remains to be shown that
$
 | f(q') - f(p) | \leq C\mu $ and $  | f(q'') - f(p) | \leq C\mu
$
for some constant $C$.
Since $q''\in \gamma_4\subset \gamma_3$, we have $ | f(q'') - f(p) | \leq \sigma (1 + 2 \sigma^3) \mu$ by part (3) in Step 2. But also $q'\in \gamma_4$, so that for the same reason $ | f(q') - f(p) | \leq \sigma (1 + 2 \sigma^3) \mu$. 
This completes the proof of Proposition \ref{prop:main2}.

\section{Proof of Theorem \ref{cor:1}}

We prove Theorem \ref{cor:1} by approximating $f$ by linear mappings at points where $f$ is differentiable.

\begin{proof} Let the assumptions of Theorem \ref{cor:1} be satisfied. Theorem  \ref{thm:main} implies that $f$ is locally quasiconformal and hence differentiable with $f_z \not= 0$  at almost every point of $U$. Let $z_0$ be a point of differentiability such that ${\hbox{\rm skew}}(f,z_0)\leq\sigma$.  We compute the maximum possible value for $H(z_0)$, which yields the same upper bound for the dilatation of $f$ at $z_0$.  Since $H(z_0)$ is invariant under M\"{o}bius transformations we may compose with translations, a dilation and a rotation to assume that  $z_0=f(z_0)=0$, $f_z(z_0)=1$ and $f_{\bar{z}}(z_0)=|f_{\bar{z}}(z_0)|<1$.  Then $f(z)=z+f_{\bar{z}}(z_0)\bar{z}+\epsilon(z)$ where $\epsilon(z)/|z|$ tends to $0$ as $z$ tends to $z_0$, and thus ${\hbox{\rm skew}}(f,z_0)={\hbox{\rm skew}}(\tilde{f},z_0)$ where $\tilde{f}(z)=z+f_{\bar{z}}(z_0)\bar{z}$. It now suffices to prove that for all affine mappings $\tilde{f}$ that can arise in this way, under the assumptions of Theorem \ref{cor:1}, the maximal dilatation of $\tilde{f}$ is at most $K(\sigma)$. This then shows that also the maximal dilatation of our original mapping $f$ is at most $K(\sigma)$, as required. Thus from now on we only consider $\tilde{f}$. 

We will first compute the skew of $\tilde{f}$, and then express $K(\tilde{f})$ in terms of ${\hbox{\rm skew}}(\tilde{f})$.

Note that $|\tilde{f}(a)-\tilde{f}(b)|=|\tilde{f}(a+v)-\tilde{f}(b+v)|$, $|\tilde{f}(a)-\tilde{f}(b)|=|\tilde{f}(\bar{a})-\tilde{f}(\bar{b})|$ and
$|\tilde{f}(a)-\tilde{f}(b)|/|\tilde{f}(a)-\tilde{f}(c)|=|\tilde{f}(ra)-\tilde{f}(rb)|/|\tilde{f}(ra)-\tilde{f}(rc)|$
for all $a,b,c,v\in\mathbb{C}$ with $a\neq c$ and all $r>0$.  This implies ${\hbox{\rm skew}}(\tilde{f}(T))$ where $T$ is an equilateral triangle is invariant under translations, complex conjugation and dilations of $T$.  Thus for all equilateral triangles $T$,  
$${\hbox{\rm skew}}(\tilde{f}(T))\in \left\{\frac{|\tilde{f}(z)-\tilde{f}(0)|}{|\tilde{f}(ze^{ i\pi/3})-\tilde{f}(0)|}: |z|=1\right\}.$$ 
Indeed, suppose $T$ has vertices $A$, $B$ and $C$, and ${\hbox{\rm skew}}(T)=\frac{|\tilde{f}(A)-\tilde{f}(B)|}{|\tilde{f}(A)-\tilde{f}(C)|}$.  First we translate $A$ to the origin, and then we dilate $T$ so its side lengths are equal to $1$.   If $\overline{AB}$ is $\pi/3$ radians clockwise from $\overline{AC}$, it is clear that our statement is true; otherwise we take the complex conjugate of $T$ to change the orientation of $T$ and then, since ${\hbox{\rm skew}} ( \tilde{f} (T) )$ is invariant under complex conjugation of $T$, our claim is true.

Hence we have
$${\hbox{\rm skew}}(\tilde{f})=\max\left\{ \frac{|\tilde{f}(z)-\tilde{f}(0)|}{|\tilde{f}(ze^{ i\pi/3})-\tilde{f}(0)|}: |z|=1\right\}
=\max\left\{ \frac{|\tilde{f}(z)|}{|\tilde{f}(ze^{ i\pi/3})|}: |z|=1\right\}.$$

Let $\mu=f_{\bar{z}}(z_0)$. If $\mu=0$, then ${\hbox{\rm skew}} ( \tilde{f}  ) =1$, so we  assume that $0<\mu<1$. Write 
$\nu=\mu+\mu^{-1} > 2$ and $\beta = e^{i\pi/6}$.
Let $w\in{\mbox{$\mathbb C$}}$ with $|w|=1$.
We have $$|\tilde{f}(w)|^2=|w+\mu\bar{w}|^2
=(w+\mu\bar{w} )(\bar w+\mu w) =  1+\mu^2 + \mu(w^2+\bar w^2)= \mu [\nu + (w^2+\bar w^2)]\,.$$

Now we are able to maximize $ |\tilde{f}(\beta w )|/|\tilde{f}(\bar \beta w)|$ with respect to $w$.  
Set $z=w^2$ and  $\alpha=e^{i\pi/3}$.
Since we have assumed $|w|=1$, we can instead  maximize 
$$\kappa= \left| \frac{\tilde{f}(\beta w )}{ \tilde{f}(\bar \beta w)}\right |^2=   \frac{\nu + \alpha z+\bar \alpha \bar z }{\nu + \bar \alpha z+\alpha \bar z  }\,. $$
We write $z=e^{i x}$, $x \in{\mbox{$\mathbb R$}}$,   so that $z' \equiv dz/dx = i z$, $\bar z'= -i\bar z$. We may differentiate $\kappa$ as a function of $x$. It follows
that $\kappa'=0$ if, and only if,
$$ (\alpha z-\bar \alpha \bar z )( \nu + \bar \alpha z+\alpha \bar z ) -  (\nu + \alpha z+\bar \alpha \bar z )(\bar \alpha z- \alpha \bar z )=0\,.$$
Thus
$$ \nu ( \alpha z-\bar \alpha \bar z-\bar \alpha z+\alpha \bar z)=  z^2 - \alpha^2 +  \bar \alpha ^2 - \bar z^2 -z^2 - \alpha^2 + \bar \alpha^2 +\bar z^2$$
which is equivalent to
$$\nu(z+\bar z)(\alpha-\bar \alpha)= 2(\bar \alpha -\alpha)(\bar \alpha +\alpha)\,.$$
Therefore $$\cos x =  -\frac{2}{\nu} \cos \frac{\pi}{3}=  -\frac{1}{\nu}\,.$$
It follows that $\sin^2 x = 1 -1/\nu^2$ so $\kappa'=0$ for $$z= \frac{1}{\nu}\left( -1+ i\varepsilon \sqrt{\nu^2 -1}\right)$$
with $\varepsilon \in\{\pm 1\}$.
For these values of $z$, one gets
$$\kappa =    \frac{\nu +  2 Re (\alpha z)}{\nu + 2 Re (\bar \alpha z) } =      \frac{ \nu^2 -1 - \varepsilon \sqrt{3(\nu^2-1)}}{ \nu^2 -1 + \varepsilon \sqrt{3(\nu^2-1)}}$$
which is maximal for $\varepsilon =-1$. 
So we obtain $${\hbox{\rm skew}}(\tilde{f})^2 = \frac{ \nu^2 -1 + \sqrt{3(\nu^2-1)}}{ \nu^2 -1 -\sqrt{3(\nu^2-1)}} = \frac{\sqrt{(\nu^2-1)/3} + 1}{\sqrt{(\nu^2-1)/3} - 1} > 1 \,.$$
Note that $\nu^2-1>3$ since $\nu>2$. 
Let us write $\tau= {\hbox{\rm skew}}(\tilde{f})>1$ so that
$$\sqrt{(\nu^2-1)/3}= \frac{\tau^2+1}{\tau^2 -1},\ \nu^2 = 3 \left(\frac{\tau^2+1}{\tau^2 -1}\right)^2+1 = \frac{4(\tau^4+\tau^2 +1)}{(\tau^2-1)^2}$$
and thus $$\mu + \mu^{-1} = \nu = \frac{2 \sqrt{\tau^4+\tau^2 +1}}{\tau^2-1}\,.$$
Hence $$\mu^2 - 2 \mu \frac{\sqrt{\tau^4+\tau^2 +1}}{\tau^2-1} + 1 =0\,.$$
We compute the reduced discriminant $$\Delta'= \frac{\tau^4+\tau^2 +1}{(\tau^2-1)^2} - 1=  \frac{3\tau^2}{(\tau^2-1)^2}$$
and we deduce from $0<\mu<1$ that $$\mu= \frac{\sqrt{\tau^4+\tau^2 +1} -\sqrt{3}\tau }{\tau^2-1}\,.$$

Thus $$K(\tilde{f})= \frac{1+\mu}{1-\mu}= \frac{ \tau^2-1 +\sqrt{\tau^4+\tau^2 +1} -\sqrt{3}\tau }{\tau^2-1 -\sqrt{\tau^4+\tau^2 +1} +\sqrt{3}\tau }\, = \varphi(\tau) ,$$
say.
Write $t= \sqrt{\tau^4+\tau^2+1} $ and note  that $t^2= (\tau^2-1)^2 + 3\tau^2$ so that
\begin{eqnarray*}
\varphi(\tau) & = & \frac{[(\tau^2 - 1 + t) - \sqrt{3} \tau] [(\tau^2 - 1 + t ) + \sqrt{3} \tau ]}{[(\tau^2 - 1) +  \sqrt{3} \tau] ^2 -t^2}\\
&= &  \frac{ [(\tau^2 - 1) + t]^2  - 3\tau^2}{2\sqrt{3}\tau (\tau^2 -1)}\  =  \  \frac{2[(\tau^2 - 1)^2 + t(\tau^2-1)]}{2\sqrt{3}\tau (\tau^2 -1)}\\
& =& \frac{\tau^2-1+\sqrt{\tau^4+\tau^2+1}}{\sqrt{3}\tau}\,.\end{eqnarray*}

 Differentiation shows that $\varphi(\tau)$ is an increasing function of $\tau$, so that  
 since  $\tau = {\hbox{\rm skew}}(\tilde{f}) = {\hbox{\rm skew}}(f,z_0)\leq \sigma$, 
we have
$$
K(\tilde{f}) = \varphi(\tau) \leq 
\varphi(\sigma) = \frac{\sigma^2-1+\sqrt{\sigma^4+\sigma^2+1}}{\sqrt{3}\sigma}\,.$$
Hence $K(f)\leq \varphi(\sigma)=K(\sigma)$ as defined in Theorem~\ref{cor:1}.

If $\mu\in (0,1)$ is given and if $f(z)=z+\mu \overline{z}$, then we may take $\tilde{f}=f$ and $\sigma=\tau$ in the above argument, and we see that $K(f)=K(\sigma)$. Thus the upper bound in Theorem~\ref{cor:1} is best possible.
\end{proof}

\section{An Alternative Proof of the Quasiconformality of Mappings Satisfying the Hypotheses of Theorem \ref{thm:main}} \label{ansection}

From Proposition \ref{prop:main1}, there are several ways of establishing that  a mapping $f$ satisfying the hypotheses of Theorem \ref{thm:main}
satisfies the analytic definition of quasiconformality which is equivalent to Definition \ref{metricdefn}. For $ | \xi | = 1$, we denote by $\partial_{\xi} f(z)$ the directional derivative $\lim_{t \downarrow 0} ( f(z+t \xi ) -f(z) )/( t \xi ) $. 

\begin{definition} \label{def2}
We say that a homeomorphism $f:U\to V$ is absolutely continuous on lines if for every rectangle $R=\{(x,y):a<x<b, c<y<d\}$ with $\overline{R}\subset U$, f is absolutely continuous on a.e.~ interval $I_x=\{(x,y):c<y<d\}$ and a.e.~ interval $I_y=\{(x,y): a<x<b\}$.  An orientation-preserving homeomorphism $f$ is called quasiconformal if $f$ is absolutely continuous on lines and there exists $K\geq 1$ such that $$\max_{\xi}|\partial_{\xi} f(z)|\leq K\min_{\xi}|\partial_{\xi} f(z)| \text{ a.e.}$$
\end{definition}

Proposition \ref{prop:main1} tells us that the image of every equilateral triangle, $T$, contains a disk with radius proportional to $L(f(T))$.  
In \cite[Section 4.5]{MR2245223}, Hubbard uses this to prove that the map belongs to the Sobolev space $W^{1,2}_{loc}$ by an approximation argument. 
We propose another approach which
shows directly that the map satisfies the  ACL property.

\begin{proof}
First we show that $f$ is absolutely continuous on lines.  This part of the proof parallels Pfluger's proof that a mapping satisfying the geometric definition of quasiconformality is absolutely continuous on lines.  His proof can be found in \cite{MR0101307} and is reproduced in English in \cite{MR0344463}, p.~162.  We fix a rectangle $R=\{(x,y):a<x<b,\, c<y<d\}$ and let $I_y=\{(x,y):a<x<b\}$ for $y$ between $c$ and $d$.  Define $A(y)$ to be the area in $f(R)$ beneath the image of $I_y$.  Since $A$ is an increasing function of $y$, it is differentiable almost everywhere.  We will show that $f|_{I_y}$ is absolutely continuous for all $y$ at which $A$ is differentiable; a similar argument applies to vertical line segments.  We select an arbitrary collection $\{(z_k^*,z_k)\}_{k=1}^n$ of disjoint sub-intervals of $I_y$ where $z_k=(x_k,y)$ and $z_k^*=(x_k^*,y)$.  

We consider the collection of rectangles $\{R_k\}_{k=1}^n$ where each $R_k$ has height $\delta$ and has its bottom side contained in the $k$th sub-interval.  More precisely, fix $k$ and draw a rectangle of height $\delta$ above $(z_k^*,z_k)$.  Set $N_k$ equal to the smallest integer  greater than or equal to $ \frac{|x_k^*-x_k|}{\delta}$.  We  draw $N_k$ equilateral triangles in $R_k$.  The triangles  all have one side on the interval $[z_k^*,z_k]$ and they  overlap only at their vertices.  The first $N_k-1$ triangles  have width $\delta$ and the last triangle has  width $x_k^*-x_k-\delta(N_k-1)$. 
Let $\Delta_{k,i}$ denote the side length of the $i$th triangle for $1\leq i\leq N_k$ and set $\Delta_{k,0}=0$.  Set $N=\sum_{k=1}^n N_k$.  Then 
$N\leq \frac{\sum_{k=1}^n |x_k^*-x_k|}{\delta}+n$.  

Let $\alpha$ be as in Proposition \ref{prop:main1}. By Proposition \ref{prop:main1}, the image of each of our triangles must contain a disk of radius comparable to the greatest distance between the images of the vertices of the triangle.  Thus the total area of all the images of our rectangles of height $\delta$ and width $x_k^*-x_k$ is greater than or equal to
\begin{equation}\label{rectangleinequality}
\sum_{k=1}^n\sum_{j=1}^{N_k} \pi \alpha^2 \left(    | f ( x_k+\sum_{i=1}^j \Delta_{k,i},y ) - f( x_k+\sum_{i=1}^j \Delta_{k,{i-1}},y ) |   \right)^2.
\end{equation}
Then by the Cauchy--Schwarz inequality, (\ref{rectangleinequality}) is greater than or equal to
\begin{eqnarray*}
&{}& 
\frac{\pi \alpha^2  }{N } \left(\sum_{k=1}^n\sum_{j=1}^{N_k}|f(x_k+\sum_{i=1}^j \Delta_{k,i},y)-f(x_k+\sum_{i=1}^j \Delta_{k,{i-1}},y)|\right)^2
\\
&\geq& \frac{\pi \alpha^2 }{N }\left(\sum_{k=1}^n|f(z_k^*)-f(z_k)|\right)^2 
\geq
\frac{\pi \alpha^2 }{\left(\frac{\sum_{k=1}^n |x_k^*-x_k|}{\delta}+n\right)}\left(\sum_{k=1}^n|f(z_k^*)-f(z_k)|\right)^2
\\
&=& \frac{\pi \alpha^2 }{\left(\frac{\sum_{k=1}^n |x_k^*-x_k|+\delta n}{\delta}\right)}\left(\sum_{k=1}^n|f(z_k^*)-f(z_k)|\right)^2 .
\end{eqnarray*}

Recall that $A(y)$ is defined to be the area in $f(R)$ beneath the image of the line segment $I_y$.   Since the sum of the areas of the images of our rectangles of height $\delta$ and width $x_k^*-x_k$ is less than or equal to $A(y+\delta)-A(y)$, we have
$$ \pi \alpha^2  \left(\sum_{k=1}^n|f(z_k^*)-f(z_k)|\right)^2\leq \left( \frac{ A(y+\delta)-A(y)}{\delta}\right)
\left(\sum_{k=1}^n |x_k^*-x_k|+\delta n\right).$$
Since we chose $y$ at which $A$ is differentiable, letting $\delta\to 0$  gives
$$\pi \alpha^2  \left(\sum_{k=1}^n|f(z_k^*)-f(z_k)|\right)^2\leq A'(y)\left(\sum_{k=1}^n |x_k^*-x_k|\right).$$
Since $A'(y)$ exists almost everywhere this gives absolute continuity on almost every horizontal line segment.  
The proof is analogous for vertical line segments.

To then conclude that $f$ is quasiconformal we  note that since $f$ is open and absolutely continuous on lines, $f$ is differentiable almost everywhere by a theorem of Gehring and Lehto \cite{MR2472875}. 
Now the  computations in the proof of Theorem \ref{cor:1}  show  that
 $$\max_{\xi}|\partial_{\xi} f(z)|\leq K(\sigma) \min_{\xi}|\partial_{\xi} f(z)|$$
 at points where $f$ is differentiable and hence almost everywhere. According to the analytic definition of quasiconformality (Definition~\ref{def2}), the mapping $f$ is quasiconformal.
\end{proof}

\section{Appendix by Colleen Ackermann: An Analogue of the Main Theorem in Hilbert Spaces of Dimension at Least Three}

In dimensions three and higher the proof of an analogue of Theorem \ref{thm:main} is surprisingly simpler than the proof of Theorem \ref{thm:main}.  Furthermore the proof itself gives an elegant bound on $K(\sigma)$.

\begin{theorem}
Let $\mathcal{H}_1$ and $\mathcal{H}_2$ be 
Hilbert spaces with $dim(\mathcal{H}_1)=dim(\mathcal{H}_2)\geq 3$ and let $U\subset\mathcal{H}_1, V\subset\mathcal{H}_2$ be domains.  Suppose that $f:U\to V$ is a homeomorphism and that for all closed equilateral triangles $T\subset U$, $skew(f(T))\leq \sigma$.  Then $f$ is $\sigma^3$-quasiconformal when using the metric definition of quasiconformality.
\end{theorem}

Note that Definition~\ref{metricdefn} can be used to define quasiconformal mappings also between Hilbert spaces, whether finite-dimensional or infinite-dimensional.
The same discussion on the locality of quasiconformality made in the introduction applies here as well and enables us  to weaken the hypothesis that
$\hbox{skew}(f(T))\le \sigma$ holds for all equilateral triangles.

\begin{proof}
Fix a point $p\in U$, a positive number $r$ with $r<{\mbox{\rm dist}}(p, \partial U)$ and points $a_1,a_2\in \partial B(p,r)$. 
 We will prove $|f(a_1)-f(p)|\leq \sigma^3 |f(p)-f(a_2)|$.  
Let $P\subset\mathcal{H}_1$ be the affine plane containing $\{p,a_1,a_2\}$ and let us fix a point $q\in\mathcal{H}_1$ so that $\overset{\rightarrow}{pq}   $ is orthogonal to $P$ and
 let us restrict ourselves to the three-dimensional space containing $\{p,a_1,a_2,q\}$ that we may identify with $\mathbb{R}^3$. For the sake
 of convenience, we will consider spherical coordinates $(r,\theta,\varphi)$ with origin $p$  so that $(r,\theta,\pi/2)$ reduces to polar coordinates in $P$
 and so that $a_1$ and $a_2$ have coordinates
 $(r, \pm \theta/2, \pi/2)$ where $\theta\in (0,\pi]$ denotes the non-oriented angle between $\overset{\rightarrow}{pa_1}   $ and  $\overset{\rightarrow}{pa_2}   $. 
 
{\bf Case 1:}
 If  $\theta \le  2\pi/3$, then we may find $\varphi\in [0,\pi/2]$ such that $2 \sin\varphi \cos (\theta/2) =1$ and we
 set $b= (r, 0, \varphi)$.  Both  triangles $T_j$ with vertices $\{p,a_j,b\}$, $j=1,2$, are equilateral by construction and they 
share a common side with end points at $b$ and $p$. 
 
 Thus
$$|f(p)-f(a_1)|\leq \sigma|f(p)-f(b)|\leq \sigma^2|f(p)-f(a_2)|.$$

{\bf Case 2:}
If  $\theta > 2 \pi/3$,  consider the equilateral triangle $T_0$ with vertices $p$, $a_1$ and $b'$ where $b'$ is the image of $a_1$ under a rotation in $P$ of angle $\pi/3$
so that the smaller angle between $b'$ and $a_2$ is less than or equal to $2\pi/3$.  Thus by Case 1 
$$|f(p)-f(b')|\leq \sigma^2 |f(p)-f(a_2)|.$$
Then since the triangle $T_0$ has sides with endpoints at $p$ and $a_1$, and $p$ and $b'$ we have
$$|f(p)-f(a_1)|\leq \sigma |f(p)-f(b')|\leq \sigma^3 |f(p)-f(a_2)|.$$
\end{proof}

\section{Acknowledgments}

We would like to thank the referee for several very helpful comments that improved the clarity of the exposition. The authors would also like to thank Jeremy Tyson for many productive discussions particularly with regard to the Appendix. We would further like to thank Kari Astala for helpful remarks related to the formulation of Theorem \ref{cor:1}.

The key point of both of our proofs of Theorem \ref{thm:main}  given here is contained in Proposition
 \ref{prop:main1},
where it is proved that the image of a triangle  contains a disk of a
definite size,
exhibiting a certain length-area estimate. This approach goes back to
Pfluger and was pushed forward
by P. Koskela and S. Rogovin who proved that the ACL property of a
homeomorphism $f$
between open sets of ${\mbox{$\mathbb R$}}^n$, $n\ge 2$, could be established from an
$L^1$-control of
    \begin{equation}\label{KRInequality}
  k_f=\liminf_{r\to
0}\left(\frac{{\mbox{\rm diam}}(f(D(x,r)))^n}{|f(D(x,r))|}\right)^{1/(n-1)}
  \end{equation}
   where $|f(D(x,r))|$ denotes the Lebesgue measure of $f(D(x,r))$; see
\cite{MR2173371} for details.
  The authors would like to thank Pekka Koskela for mentioning this
similarity.

\end{document}